\documentclass[12pt,oneside]{amsart}

\usepackage{amssymb,amsxtra,mathtools}

\usepackage[
  paper=a4paper,
  headsep=20pt,
  footskip=18pt,
  text={160mm,240mm},
  centering,
  includehead
]{geometry}

\usepackage{newtxtext,newtxmath}

\usepackage[bookmarks,bookmarksdepth=2,pdfencoding=auto,psdextra]{hyperref}
\usepackage{enumitem}
\setlist[enumerate,1]{label=(\arabic*),font=\upshape}
\usepackage[capitalize,noabbrev]{cleveref}

\crefname{enumi}{}{}
\crefname{equation}{}{}
\usepackage{keytheorems}
\usepackage{graphicx,xcolor}
\usepackage{tikz}
\usetikzlibrary{cd,calc,arrows,decorations.pathreplacing}
\tikzcdset{arrow style=math font}
\usetikzlibrary{arrows,spath3,intersections,nfold}
\makeatletter
\tikzset{
  offset now/.code=
    \pgfgetpath\tikz@temp\pgfsetpath\pgfutil@empty\pgfoffsetpath\tikz@temp{#1}
}
\makeatother

\iftrue
  \makeatletter
  \def\@settitle{%
    \begin{flushleft}%
      \LARGE\bfseries
      \strut\@title\strut
    \end{flushleft}%
  }
  \def\@setauthors{%
    \begingroup
    \def\thanks{\protect\thanks@warning}%
    \trivlist
    \raggedright
    \large \@topsep30\p@\relax
    \advance\@topsep by -\baselineskip
    \item\relax
    \author@andify\authors
    \def\\{\protect\linebreak}%
    \authors
    \ifx\@empty\contribs
    \else
    ,\penalty-3 \space \@setcontribs
    \@closetoccontribs
    \fi
    \normalfont
    \endtrivlist
    \endgroup
  }
  \def\@setaddresses{\par
    \nobreak \begingroup
    \small\raggedright
    \def\author##1{\nobreak\addvspace\smallskipamount}%
    \def\\{\unskip, \ignorespaces}%
    \interlinepenalty\@M
    \def\address##1##2{\begingroup
      \par\addvspace\bigskipamount\noindent
      \@ifnotempty{##1}{(\ignorespaces##1\unskip) }%
      {\ignorespaces##2}\par\endgroup}%
    \def\curraddr##1##2{\begingroup
      \@ifnotempty{##2}{\nobreak\noindent\curraddrname
        \@ifnotempty{##1}{, \ignorespaces##1\unskip}\/:\space
        ##2\par}\endgroup}%
    \def\email##1##2{\begingroup
      \@ifnotempty{##2}{\nobreak\noindent E-mail address%
        \@ifnotempty{##1}{, \ignorespaces##1\unskip}\/:\space
        \ttfamily##2\par}\endgroup}%
    \def\urladdr##1##2{\begingroup
      \def~{\char`\~}%
      \@ifnotempty{##2}{\nobreak\noindent\urladdrname
        \@ifnotempty{##1}{, \ignorespaces##1\unskip}\/:\space
        \ttfamily##2\par}\endgroup}%
    \addresses
    \endgroup
    \global\let\addresses=\@empty
  }
  \def\@setabstracta{%
    \ifvoid\abstractbox
    \else
    \skip@30pt \advance\skip@-\lastskip
    \advance\skip@-\baselineskip \vskip\skip@
    \box\abstractbox
    \prevdepth\z@ 
    \vskip-20pt
    \fi
  }
  \renewenvironment{abstract}{%
    \ifx\maketitle\relax
    \ClassWarning{\@classname}{Abstract should precede
      \protect\maketitle\space in AMS document classes; reported}%
    \fi
    \global\setbox\abstractbox=\vtop \bgroup
    \normalfont\small
    \list{}{\labelwidth\z@
      \leftmargin0pc \rightmargin\leftmargin
      \listparindent\normalparindent \itemindent\z@
      \parsep\z@ \@plus\p@
      
    }%
  \item[\hskip\labelsep\bfseries\abstractname.]%
  }{%
    \endlist\egroup
    \ifx\@setabstract\relax \@setabstracta \fi
  }

  \def\ps@headings{\ps@empty
    \def\@evenhead{%
      \setTrue{runhead}%
      \normalfont\scriptsize
      \rlap{\thepage}\hfill
      \def\thanks{\protect\thanks@warning}%
      \leftmark{}{}}%
    \def\@oddhead{%
      \setTrue{runhead}%
      \normalfont\scriptsize
      \def\thanks{\protect\thanks@warning}%
      \rightmark{}{}\hfill \llap{\thepage}}%
    \let\@mkboth\markboth
  }\ps@headings

  \expandafter\def\csname section\endcsname{\@startsection{section}{1}%
    \z@{-1.4\linespacing\@plus-.5\linespacing}{.8\linespacing\@plus.3\linespacing}%
    {\normalfont\bfseries\Large}}
  \expandafter\def\csname subsection\endcsname{\@startsection{subsection}{2}%
    \z@{-.8\linespacing\@plus-.3\linespacing}{.5\linespacing\@plus.3\linespacing}%
    {\normalfont\bfseries\large}}
  \expandafter\def\csname subsubsection\endcsname{\@startsection{subsubsection}{3}%
    \z@{-.7\linespacing\@plus-.2\linespacing}{.3\linespacing\@plus.2\linespacing}%
    {\normalfont\bfseries\normalsize}}
  \expandafter\def\csname paragraph\endcsname{\@startsection{paragraph}{4}%
    \z@{.7\linespacing\@plus.2\linespacing}{-1.5ex}%
    {\normalfont\itshape}}
  \def\@secnumfont{\bfseries}

  \renewcommand\contentsnamefont{\bfseries}
  \def\@starttoc#1#2{\begingroup
    \setTrue{#1}%
    \par\removelastskip\vskip\z@skip
    \@startsection{}\@M\z@{\linespacing\@plus\linespacing}%
    {.5\linespacing}{
      \contentsnamefont}{#2}%
    \ifx\contentsname#2%
    \else \addcontentsline{toc}{section}{#2}\fi
    \makeatletter
    \@input{\jobname.#1}%
    \if@filesw
    \@xp\newwrite\csname tf@#1\endcsname
    \immediate\@xp\openout\csname tf@#1\endcsname \jobname.#1\relax
    \fi
    \global\@nobreakfalse \endgroup
    \addvspace{32\p@\@plus14\p@}%
    \let\tableofcontents\rela\x
  }
  \def\contentsname{Contents}
  \def\l@section{\@tocline{2}{.5ex}{0mm}{5pc}{}}
  \def\l@subsection{\@tocline{2}{0pt}{2em}{5pc}{}}
  \makeatother

\fi

\def\to{\mathchoice{\longrightarrow}{\rightarrow}{\rightarrow}{\rightarrow}}
\makeatletter
\newcommand{\shortxra}[2][]{\ext@arrow 0359\rightarrowfill@{#1}{#2}}
\def\longrightarrowfill@{\arrowfill@\relbar\relbar\longrightarrow}
\newcommand{\longxra}[2][]{\ext@arrow 0359\longrightarrowfill@{#1}{#2}}

\makeatother

\makeatletter
\def\addtagsub#1{\let\oldtf=\tagform@\def\tagform@##1{\oldtf{##1}\hbox{$_{#1}$}}}
\makeatother

\makeatletter
\def\Nopagebreak{\@nobreaktrue\nopagebreak}
\makeatother

\foreach \n in 
{Z,Q,R,C}
{\expandafter\xdef\csname\n\endcsname{\noexpand\mathbb{\n}}}

\foreach \n in 
{A,B,F,S}
{\expandafter\xdef\csname b\n\endcsname{\noexpand\mathbb{\n}}}

\foreach \n in 
{A,...,Z}
{\expandafter\xdef\csname c\n\endcsname{\noexpand\mathcal{\n}}}

\foreach \n in 
{dim,rk,rank,sign,sgn,Ker,Coker,Im,Hom,End,GL,SO,
 Tor,Ext,Gal,Irr,dis,Wh,Hopf,tr,Tr,vol,colim,grk,diam,
 lk,st,Mod,Aut,Map,SL,Dax,dax,FQ,Homeo,Diff,Emb,ks,
 CAT,TOP,DIFF,PL,O}
{\expandafter\xdef\csname\n\endcsname{\noexpand\operatorname{\n}}}

\foreach \n/\s in 
{inte/int}
{\expandafter\xdef\csname\n\endcsname{\noexpand\operatorname{\s}}}

\foreach \n in 
{id,pr,PI,Pr,ad}
{\expandafter\xdef\csname\n\endcsname{\noexpand\mathrm{\n}}}

\foreach \n/\s in 
{csum/\#,bcsum/\natural/}
{\expandafter\xdef\csname\n\endcsname{\noexpand\mathbin{\s}}}

\def\rel{\ifmmode\;{\mathgroup\symoperators rel}\;\else rel\ \fi}
\def\sm{\smallsetminus}

\newkeytheorem{theoremalpha}[
  name=Theorem,refname={Theorem,Theorems},counter-format=\Alph*
]
\newkeytheorem{corollaryalpha}[
  name=Corollary,refname={Corollary,Corollaries},sibling=theoremalpha
]
\newkeytheorem{conjecturealpha}[name=Conjecture,counter-format=\Alph*]
\newkeytheorem{theorem}[numberwithin=section]
\newkeytheorem{
  question,corollary,proposition,lemma,conjecture,observation
}[sibling=theorem]
\newkeytheorem{
  definition,example,remark,notation,case
}[style=definition,sibling=theorem]
\newkeytheorem{assertion,claim}[numbered=no]
\newkeytheorem{question*}[numbered=no,name=Question]
\newkeytheorem{convention}[style=definition,numbered=no]

\newkeytheoremstyle{style-theorem-named}{
  noteseparator={},notefont=\bfseries,notebraces={}{},headformat={\NOTE}
}
\newkeytheorem{theorem-named}[style=style-theorem-named]

\newkeytheoremstyle{style-definition-named}{
  inherit-style=definition,
  noteseparator={},notefont=\bfseries,notebraces={}{},headformat={\NOTE}
}
\newkeytheorem{definition-named,conjecture-named}[style=style-definition-named]

\numberwithin{equation}{section}

\def\labelignored#1{\ignorespaces}

\begin{document}

\title
{A homological approach to Dax invariants}

\author{Jae Choon Cha}
\address{
  Center for Research in Topology and Department of Mathematics\\%
  POSTECH\\%
  Pohang Gyeongbuk 37673\\%
  Republic of Korea
}
\email{jccha@postech.ac.kr}



\begin{abstract}
  We present a homological construction of Dax invariants in 4-manifolds that is intrinsic to the topological category and requires neither smoothing nor generic representatives.
  Applications include an extension of the invariant of Cha--Kim to neat disks without assuming local flatness, and a topological definition of the Dax map for arc spaces that yields a short alternative proof of the topological Dax isomorphism theorem.
\end{abstract}

\maketitle

\section{Introduction}

The Dax invariant is a useful tool in the study of isotopy problems in 4-dimensional topology.
It originates in Dax's study of smooth embedding spaces in the metastable range~\cite{Dax:1972-1}, and Gabai's work on smooth isotopy of disks~\cite{Gabai:2021-1} initiated recent applications in dimension~4.
For developments and applications in the smooth category, see, in particular, the work of Kosanovi\'c and Teichner \cite{Kosanovic-Teichner:2024-2}, Kosanovi\'c \cite{Kosanovic:2024-1}, Schwartz \cite{Schwartz:2021-1} and Lin, Wu, Xie and Zhang \cite{Lin-Wu-Xie-Zhang:2026-1}.
While smoothness is essential in the original construction, topological versions of the invariant have been developed in~\cite{Cha-Kim:2023-1,Lin-Xie-Zhang:2026-2}.
We review these smooth and topological developments below.

In this paper, we develop a new \emph{homological approach} to Dax invariants in 4-manifolds.
Our definitions and proofs of invariance are intrinsic to the topological category and do not use smooth techniques.

In particular, our method gives a topological Dax invariant for embedded disks without assuming local flatness.
Throughout this paper, unless stated otherwise, embedded disks, surfaces and their isotopies are not assumed to be locally flat.
In particular, an isotopy means a continuous family of embeddings.
A locally flat isotopy is understood in the sense of Edwards and Kirby~\cite{Edwards-Kirby:1971-1}.
All isotopies are rel boundary, i.e., they fix the boundary pointwise.

For neat disks $D_0$ and $D_1$ that are homotopic rel $\partial$ in an oriented 4-manifold~$M$, we define the Dax invariant $\Dax(D_0,D_1)$ in \cref{subsection:dax-for-disk-definition-properties}.
We prove the isotopy invariance of our invariant in~\cref{theorem:dax-disk-isotopy}, and prove that it agrees with the invariant in~\cite{Cha-Kim:2023-1} for locally flat disks in \cref{theorem:dax-disk-topological}.

\begin{theorem-named}[\cref{theorem:dax-disk-isotopy}]
  The value of $\Dax(D_0,D_1)$ is invariant under isotopy of~$D_0$ and~$D_1$.
  If $D_0$ and $D_1$ are isotopic, then $\Dax(D_0,D_1)=0$.
\end{theorem-named}

See \cref{section:dax-for-disks} for further details on the invariants for disks.

In \cref{corollary:disk-isotopy-implies-locally-flat-isotopy} below, we record a consequence of the invariance under non-locally flat isotopy together with earlier results in~\cite{Cha-Kim:2023-1}.
We say that a neat disk $D\subset M$ has a \emph{boundary geometric dual} if there is a map $G\colon S^2\to \partial M$ which intersects $\partial D$ at exactly one point, transversely.
We remark that $G$ can be assumed to be an embedding by the loop theorem~\cite[Lemma~2.7]{Cha-Kim:2023-1}.

\begin{theorem-named}[\cref{corollary:disk-isotopy-implies-locally-flat-isotopy}]
  Let $M$ be an orientable topological 4-manifold with nonempty boundary.
  For locally flat neat disks in $M$ admitting a boundary geometric dual, isotopy implies locally flat isotopy.
  If $M$ is smooth, then for smooth neat disks in $M$ admitting a geometric dual in the boundary, topological isotopy implies smooth isotopy.
\end{theorem-named}

Another consequence concerns \emph{local knotting}.
In general, local knotting produces surfaces which are isotopic but not necessarily ambient isotopic;
in many cases it can be detected by the fundamental group of the complement.
In both smooth and topological locally flat categories, \cref{corollary:local-knotting-implies-locally-flat-isotopy} shows that local knotting on any connected neat surface, possibly nonorientable, in an orientable 4-manifold produces ambient isotopic surfaces if the surface has a boundary geometric dual.

\paragraph{Previous developments of the Dax invariant}

In the smooth formulation in~\cite{Dax:1972-1,Gabai:2021-1,Kosanovic-Teichner:2024-2,Cha-Kim:2023-1}, the Dax invariant is essentially a $\pi_1$-labeled signed count of self-intersection double points.
In this formulation, we consider a map $H\colon I^3 \to M$ satisfying a certain smooth boundary condition.
(Specifically, one can use a smooth analog of our boundary condition in \cref{section:homological-dax}.)
In particular, a null-homotopy of a loop of arc embeddings in the space of arc maps and a homotopy between embedded disks give such a map~$H$.
A key ingredient is Dax's perturbation \cite[Chapters III and~VII]{Dax:1972-1}, which is built on Thom's multijet transversality.
It homotopes $H$ rel $\partial I^3$ so that the associated \emph{track} $\overline H\colon I^3 \to M\times I^2$ defined by $\overline H(s,t,u)=(H(s,t,u),t,u)$ is a self-transverse smooth generic map.
The invariant~$\dax(H)\in \Z[\pi_1(M)\sm \{1\}]$ is defined to be the signed count of double points with nontrivial $\pi_1$-labels.
Ordering two pre-image points by their first coordinates determines the signs and $\pi_1$-labels without Wall's sheet-ordering ambiguity.
See our \cref{lemma:dax-for-generic-track} for the defining formula.
Another key ingredient is the invariance of $\dax(H)$ under homotopy of~$H$ rel~$\partial I^3$, which is proved by relating two perturbations using singularity theory for smooth maps.
In particular, both the Dax construction and the proof of homotopy invariance rely on smoothness.

A generalization to the topological category was introduced in~\cite{Cha-Kim:2023-1}, which defined the Dax invariant of homotopic locally flat disks.
They use their stable smoothing theorem for locally flat disks~\cite[Theorem~D]{Cha-Kim:2023-1} so that the smooth invariant can be applied.
For a topological 4-manifold, they use an almost smooth structure, meaning a smooth structure of the ambient 4-manifold with a singular point.
They proved that the value of the smooth invariant is independent of the choice of the stable disk smoothing, yielding a well-defined topological invariant of disks.
Independence from the chosen almost smooth structure was not settled in~\cite{Cha-Kim:2023-1}.

Recently, in \cite{Lin-Xie-Zhang:2026-2}, Lin, Xie and Zhang studied the Dax isomorphism for the Dax group $\Ker\{ \pi_1\Emb(I,M) \to \pi_1\Map(I,M) \}$, whose smooth version was presented earlier in~\cite[Theorem~3.7]{Gabai:2021-1}.
(See \cref{subsection:application-dax-isomorphism} for a precise definition of the arc spaces involved.)
For smooth 4-manifolds $M$, they showed that the topological and smooth Dax groups are isomorphic by smoothing families of arcs, and showed that the Dax isomorphism is independent of the choice of the smooth structure, using an interpretation via configuration space models from embedding calculus.
Using this, they obtained the Dax isomorphism theorem for topological 4-manifolds~\cite[Theorem~1.1]{Lin-Xie-Zhang:2026-2}.

Applications of the topological Dax invariant include isotopy smoothing for disks~\cite{Cha-Kim:2023-1}, and more generally for arbitrary surfaces~\cite{Cha-Kim:2026-1}, under a geometric dual hypothesis.
In~\cite{Cha:2026-1}, the topological Dax invariant is used to study topological mapping class groups and Torelli groups of a broad range of 4-manifolds, including various non-smoothable 4-manifolds.

\paragraph{Our homological approach}

A key ingredient we develop is a homological definition of the Dax invariant $\dax(H)$ in the topological category, without using smoothing or Dax's smooth perturbation.
Our definition applies directly to maps $H\colon I^3\to M$ whose track $\overline H\colon I^3 \to M\times I^2$ need not be self-transverse and may have infinitely many ``double points'' in general.
We show that the associated $\pi_1$-labels form a finite set, and the ordered pre-image pairs with distinct non-identity labels $g\in \pi_1(M)\sm \{1\}$ are separated by pairwise disjoint codimension-zero compact submanifolds $V_g \subset I^3\times I^3$ with boundary;
the contribution of the ordered pre-image pairs with $\pi_1$-label $g\ne 1$ is defined by taking the homological intersection of the image of the fundamental class of~$V_g$ with the diagonal in $(M\times I^2)\times (M\times I^2)$.
For details, see \cref{subsection:homological-definition}, especially \cref{definition:dax-intersection-dbint}.

We give a topological proof that $\dax(H)$ is invariant under homotopy rel~$\partial I^3$, by showing that sufficiently close maps agreeing on $\partial I^3$ have the same value of~$\dax(H)$ in \cref{subsection:homological-definition}.
We also establish useful properties of $\dax(H)$ in \cref{subsection:basic-properties} and define the Dax homomorphism of $\pi_3(M)$ in the topological category in \cref{subsection:dax-homomorphism-pi_3}.

For the special case of maps with topologically self-transverse tracks, we discuss how to compute the invariant from double points in \cref{subsection:computation-generic-track}.
From this, it follows that our topological definition generalizes the previous smooth approaches.

We consider the case of loops of arc embeddings in \cref{subsection:application-dax-isomorphism}.
We show that our homological approach also gives an intrinsic topological definition of the Dax isomorphism, yielding a short alternative proof of the result in~\cite{Lin-Xie-Zhang:2026-2} that the smooth Dax isomorphism generalizes to the topological category.

We discuss applications to invariants of disks and their isotopies in 4-manifolds in \cref{section:dax-for-disks}.
For locally flat disks, \cref{theorem:dax-disk-topological} identifies our invariant with the invariant in~\cite{Cha-Kim:2023-1}.
It follows that the latter is independent of the choice of an almost smooth structure.
The formulation of the disk invariant uses a symmetry property introduced in~\cite{Kosanovic-Teichner:2024-2} in the smooth category.
We prove a topological version in \cref{subsection:symmetry}.

\subsubsection*{Acknowledgements}

The author thanks Jianfeng Lin, Yi Xie and Boyu Zhang for discussions that motivated him to revisit the topological Dax invariant.
This work was supported by the National Research Foundation grant RS-2019-NR039996.
The author used AI tools to improve the language and readability of the manuscript.
The author takes full responsibility for its content.

\section{Topological Dax invariant}
\label{section:homological-dax}

Let $M$ be an oriented 4-manifold with nonempty boundary.
Fix two distinct points $z_0$,~$z_1 \in \partial M$.
We use $z_0$ as the basepoint.

We always assume that a map $H\colon I^3 \to M$ satisfies $H(s\times I^2)=\{z_s\}$ for $s=0,1$.
Define the \emph{track} $\overline H\colon I^3 \to M\times I^2$ by $\overline H(s,\tau)=(H(s,\tau),\tau)$ for $(s,\tau)\in I\times I^2$.
We have $\overline H(\partial I^3) \subset \partial(M\times I^2)$.
The goal of this section is to define a homological invariant generalizing Dax's self-intersection~\cite{Dax:1972-1,Gabai:2021-1,Kosanovic-Teichner:2024-1,Kosanovic-Teichner:2024-2,Cha-Kim:2023-1} for maps $H\colon I^3 \to M$ satisfying the following condition.

\def\DB{\textup{(D)}}\def\DBint{\textup{(D$_\mathrm{int}$)}}
\begin{enumerate}[label={\DB}]
  \item\label{item:dax-boundary-condition}
  $\overline H$ embeds $\partial I^3$ into $\partial(M\times I^2)$,
  i.e., $s\mapsto H(s,\tau)$ is injective for all $\tau\in\partial I^2$.
\end{enumerate}

We first consider the special case of maps satisfying the following stronger condition:

\begin{enumerate}[label={\DBint}]
  \item\label{item:dax-boundary-interior-condition}
  $\overline H$ fails to be injective only at interior points of $M\times I^2$, i.e., $p=H(s_0,\tau)=H(s_1,\tau)$ with $s_0 \ne s_1$ implies $p\in \inte(M)$ and $\tau\in \inte(I^2)$.
\end{enumerate}

Note that \cref{item:dax-boundary-interior-condition} implies~\cref{item:dax-boundary-condition}.
Also, if $H$ sends $\inte(I^3)$ to $\inte(M)$, then \cref{item:dax-boundary-condition} implies \cref{item:dax-boundary-interior-condition}.
The following perturbations are useful in our arguments.

\begin{lemma}
  \label{lemma:interior-perturb}
  \begin{enumerate}
    \item \label{item:map-interior-perturb}
    A map $H$ satisfying \cref{item:dax-boundary-condition} can be perturbed, by an arbitrarily small homotopy rel $\partial I^3$, to a map sending $\inte(I^3)$ to $\inte(M)$, and thus satisfying \cref{item:dax-boundary-interior-condition}.
    \item \label{item:homotopy-interior-perturb}
    A homotopy $\{H_v\}_{0\le v\le 1}$ rel $\partial I^3$ between maps $H_0$ and $H_1$ satisfying \cref{item:dax-boundary-interior-condition} can be perturbed, by an arbitrarily small homotopy rel $\partial I^4$, to a homotopy rel $\partial I^3$ between $H_0$ and $H_1$ through maps satisfying~\cref{item:dax-boundary-interior-condition}.
  \end{enumerate}
\end{lemma}

\begin{proof}
  Choose an arbitrarily small isotopy $\{f_t\colon M\to M\}_{0\le t\le 1}$ such that $f_0=\id_M$ and $f_t(M)\subset\inte(M)$ for $t \in(0,1]$.
  Choose a map $\lambda\colon I^3\to[0,1]$ such that $\lambda^{-1}(0)=\partial I^3$.

  To show \cref{item:map-interior-perturb}, define $H_v(x)=f_{v\lambda(x)}(H(x))$, $v\in I$.
  This is a homotopy rel~$\partial I^3$ with $H_0=H$ and $H_1(\inte(I^3))\subset\inte(M)$.
  Since $H_1$ satisfies~\cref{item:dax-boundary-condition}, $H_1$ also satisfies~\cref{item:dax-boundary-interior-condition}.

  To shoew \cref{item:homotopy-interior-perturb}, define $K_{v,w}(x)=f_{wv(1-v)\lambda(x)}(H_v(x))$, $w\in I$, and let $F_v(x)=K_{v,1}(x)$.
  Viewed as a homotopy of maps $I^3\times I\to M$ with parameter $w$, $\{K_{v,w}\}_{w\in I}$ is rel~$\partial I^4$.
  Also, $\{F_v\}_{v\in I}$ is a homotopy rel~$\partial I^3$ from $H_0$ to~$H_1$.
  For $v\in(0,1)$, we have $F_v(\inte(I^3))\subset\inte(M)$, and $\overline F_v|_{\partial I^3}=\overline H_v|_{\partial I^3}$ is an embedding into $\partial(M\times I^2)$.
  Thus $F_v$ satisfies~\cref{item:dax-boundary-interior-condition} for $v\in(0,1)$, and also for $v=0,1$ by the hypothesis.
\end{proof}

\subsection{Homological definition of the Dax self-intersection}
\label{subsection:homological-definition}

Let $s\colon I^3 \to I$ and $\tau\colon I^3 \to I^2$ be the projections defined by $x=(s(x),\tau(x))$ for $x\in I^3$.
Let $P=\{(x,y)\in I^3\times I^3:s(x)\le s(y)\}$, and $\Delta\subset (M\times I^2)^2$ be the diagonal.

Suppose that $H\colon I^3\to M$ is a map satisfying~\cref{item:dax-boundary-interior-condition}.
Let
\[
  \Phi_H = (\overline H \times \overline H) |_{P} \colon P \to (M\times I^2)^2.
\]
Note that $\Phi_H^{-1}(\Delta) = \{(x,y)\in P \mid \overline H(x)=\overline H(y)\}$.
Define $g_H\colon \Phi_H^{-1}(\Delta) \to \pi_1(M)$ by
\[
  g_H(x,y) = \bigl[ H|_{[0,s(x)]\times\tau(x)}\cdot (H|_{[0,s(y)]\times\tau(y)})^{-1} \bigr] \in \pi_1(M) = \pi_1(M,z_0).
\]
Since sufficiently close based loops in a manifold are homotopic rel basepoint, $g_H$ is locally constant.
Since $\Phi_H^{-1}(\Delta)$ is compact, it follows that $g_H(\Phi_H^{-1}(\Delta))$ is finite.
It also follows that for all $g\in\pi_1(M)$, $g_H^{-1}(g)$ is closed, and thus compact.

Let $G_H = g_H(\Phi_H^{-1}(\Delta))\sm \{1\}$, the set of nontrivial $\pi_1$ elements associated with ``double points'' of~$\overline H$.
Note that $g_H^{-1}(g) \subset \inte(P)$ for $g\in G_H$, since \cref{item:dax-boundary-interior-condition} holds and $g_H(x,y)=1$ if $s(x)=s(y)$.

We claim that, for $g\ne 1$, there exist pairwise disjoint compact codimension-zero submanifolds $V_g \subset \inte(P)$ with boundary such that $V_g\cap \Phi_H^{-1}(\Delta) = g_H^{-1}(g) \subset\inte(V_g)$.
Note that this allows $V_g=\emptyset$ for $g\notin G_H$.
In fact, setting $V_g=\emptyset$ for $g\notin G_H$, it suffices to choose $V_g$ for $g\in G_H$.
Using normality and the finiteness of $G_H$, choose pairwise disjoint open sets $U_g \subset \inte(P)$ such that $U_g \cap \Phi_H^{-1}(\Delta) = g_H^{-1}(g)$.
Cover the compact set $g_H^{-1}(g)$ by finitely many linear 6-simplices in~$U_g$.
Take $V_g$ to be a small closed PL regular neighborhood of the union of the simplices in~$U_g$.

Let $u_\Delta\in H^6((M\times I^2)^2, (M\times I^2)^2\sm\Delta)$ be the Thom class for $\Delta$ equipped with the orientation of $M\times I^2\cong \Delta$, so that its evaluation on a homology class is the algebraic intersection with~$\Delta$.
Let $[V_g] \in H_6(V_g,\partial V_g)$ be the fundamental class induced by the orientation of $I^3\times I^3$.
Define
\[
  n(g) = - \bigl\langle u_\Delta, (\Phi_H|_{V_g})_*[V_g] \bigr\rangle\in\Z.
\]
We remark that for oriented sheets $\R^3\times0$ and $0\times\R^3$ in $\R^6$ with intersection $+1$, the product $\R^3\times0\times 0\times\R^3$ intersects the oriented diagonal of $\R^6\times\R^6$ with sign~$-1$;
the sign convention for $n(g)$ is chosen so that such a positive self-intersection of $\overline H$ contributes~$+1$.

\begin{lemma}
  The integer $n(g)$ is independent of the choice of~$V_g$.
\end{lemma}

\begin{proof}
  Let $V_g$ and $V'_g$ be two choices.
  Choose a compact codimension-zero submanifold $W_g$ with boundary such that $g_H^{-1}(g)\subset\inte(W_g) \subset W_g\subset\inte(V_g)\cap\inte(V'_g)$.
  Since $\Phi_H$ avoids $\Delta$ on $V_g\sm\inte(W_g)$ and $V'_g\sm\inte(W_g)$, excision gives
  \[
    \bigl\langle u_\Delta, (\Phi_H|_{V_g})_*[V_g] \bigr\rangle
    = \bigl\langle u_\Delta, (\Phi_H|_{W_g})_*[W_g] \bigr\rangle
    = \bigl\langle u_\Delta, (\Phi_H|_{V'_g})_*[V'_g] \bigr\rangle.
    \qedhere
  \]
\end{proof}

\begin{definition}
  \label{definition:dax-intersection-dbint}
  Define the \emph{homological Dax self-intersection} of $H$ by
  \[
    \dax(H)=\sum_{g\ne 1}n(g)\cdot g\in\Z[\pi_1(M)\sm1].
  \]
\end{definition}

Note that, if $g\notin G_H$, then $n(g)=0$, even when $V_g$ is nonempty, since $\Phi_H(V_g)$ avoids~$\Delta$.
Thus the above sum has only finitely many nonzero terms.

We remark that our formulation in \cref{definition:dax-intersection-dbint} is motivated from the proof of \cite[Lemma~3.8]{Cha:2026-1}.

\begin{lemma}
  \label{lemma:dax-small-homotopy-invariance}
  Suppose $H\colon I^3\to M$ satisfies~\cref{item:dax-boundary-interior-condition}.
  Then there exists $\epsilon>0$ such that, for any homotopy $\{H_v\}_{v\in I}$ rel $\partial I^3$ with $H_0=H$ through maps that are $\epsilon$-close to $H$ and satisfy~\cref{item:dax-boundary-interior-condition}, we have $\dax(H_1)=\dax(H)$.
\end{lemma}

\begin{proof}
  We use the following assertion.
  Suppose $C\subset P$ is compact and $g_H=h\in\pi_1(M)$ on $\Phi_H^{-1}(\Delta)\cap C$.
  Then $g_F=h$ on $\Phi_F^{-1}(\Delta)\cap C$ whenever $F$ is sufficiently close to~$H$, with the required closeness depending only on $H$ and~$C$.
  Indeed, there is an open neighborhood $U$ of $\Phi_H^{-1}(\Delta)\cap C$ in $C$ such that $g_F=h$ on $\Phi_F^{-1}(\Delta)\cap U$ whenever $F$ is sufficiently close to~$H$, by uniform continuity, compactness and the fact that sufficiently close based loops are homotopic rel basepoint.
  Since $\Phi_H$ avoids $\Delta$ on the compact set $C\sm U$, $\Phi_F$ also avoids $\Delta$ on $C\sm U$ whenever $F$ is sufficiently close to~$H$.
  This proves the assertion.

  Fix $\{V_g\}_{g\in G_H}$ as in the definition of $n(g)$ for~$H$.
  Apply the above assertion to the finitely many compact sets $V_g$, for $g\in G_H$, and $C'=P\sm\bigcup_{g\in G_H}\inte(V_g)$, to choose $\epsilon>0$ such that the conclusions hold simultaneously whenever $F$ is $\epsilon$-close to~$H$.
  Since $g_H=g$ on $\Phi_H^{-1}(\Delta)\cap V_g$ and $g_H=1$ on $\Phi_H^{-1}(\Delta)\cap C'$, we have $g_F=g$ on $\Phi_F^{-1}(\Delta)\cap V_g$ and $g_F=1$ on $\Phi_F^{-1}(\Delta)\cap C'$.
  Since $P=(\bigcup_{g\in G_H} \inte (V_g)) \cup C'$, every point of $\Phi_F^{-1}(\Delta)$ lies in either $\Phi_F^{-1}(\Delta)\cap \inte(V_g)$ for some $g\in G_H$ or~$C'$.
  Therefore, $G_F \subset G_H$ and $g_F^{-1}(g)=\Phi_F^{-1}(\Delta)\cap V_g\subset\inte(V_g)$.
  This shows that $\{V_g\}_{g\in G_H}$ can be used to compute $\dax(F)$ for any $F$ satisfying \cref{item:dax-boundary-interior-condition} that is $\epsilon$-close to~$H$;
  for $g\in G_H \sm G_F$, we have $\Phi_F^{-1}(\Delta) \cap V_g = g_F^{-1}(g) = \emptyset$, i.e., $\Phi_F$ maps $V_g$ to $(M\times I^2)^2 \sm \Delta$, so that $\bigl\langle u_\Delta, (\Phi_F|_{V_g})_*[V_g] \bigr\rangle = 0$.

  Now, let $\{H_v\}_{v\in I}$ be a homotopy as in the statement, with the above choice of~$\epsilon$.
  Then, since $\Phi_{H_v}(\partial V_g)\cap\Delta=\emptyset$ for all~$v$, homotopy invariance of relative homology gives
  \[
    \bigl\langle u_\Delta,(\Phi_H|_{V_g})_*[V_g] \bigr\rangle
    =
    \bigl\langle u_\Delta,(\Phi_{H_v}|_{V_g})_*[V_g] \bigr\rangle.
  \]
  It follows that $\dax(H_v)=\dax(H)$.
\end{proof}

\begin{lemma}
  \label{lemma:dax-homotopy-invariance-dbint}
  If $H$, $F\colon I^3 \to M$ satisfy \cref{item:dax-boundary-interior-condition} and are homotopic rel~$\partial I^3$, then $\dax(H)=\dax(F)$.
\end{lemma}

\begin{proof}
  Let $\{H_r\}_{0\le r\le 1}$ be a homotopy rel $\partial I^3$ with $H_0=H$ and $H_1=F$.
  We may assume that $H_r$ satisfies \cref{item:dax-boundary-interior-condition} for all~$r$.
  By \cref{lemma:dax-small-homotopy-invariance}, the map $r\mapsto \dax(H_r)$ is locally constant, and thus $\dax(H_0)=\dax(H_1)$.
\end{proof}

The previous definition of $\dax(H)$ generalizes as follows, using \cref{lemma:dax-homotopy-invariance-dbint}.

\begin{definition}
  \label{definition:dax-intersection}
  Let $H\colon I^3 \to M$ be a map satisfying \cref{item:dax-boundary-condition}, without requiring~\cref{item:dax-boundary-interior-condition}.
  Choose a map $F\colon I^3 \to M$ satisfying \cref{item:dax-boundary-interior-condition} that is homotopic to $H$ rel~$\partial I^3$, using~\cref{lemma:interior-perturb}~\cref{item:map-interior-perturb}.
  Define $\dax(H)=\dax(F)$, where $\dax(F)$ is defined by \cref{definition:dax-intersection-dbint}.
\end{definition}

By \cref{lemma:dax-homotopy-invariance-dbint}, the value of $\dax(H)$ is independent of the choice of~$F$.

\subsection{Basic properties}
\label{subsection:basic-properties}

In what follows, we assume that given maps $I^3\to M$ always satisfy~\cref{item:dax-boundary-condition}.

\begin{theorem}[Homotopy invariance]
  \label{theorem:dax-homotopy-invariance}
  If $H$,~$F$ are homotopic rel~$\partial I^3$, then $\dax(H)=\dax(F)$.
\end{theorem}

\begin{proof}
  It is an immediate consequence of \cref{definition:dax-intersection,lemma:dax-homotopy-invariance-dbint}.
\end{proof}

\begin{remark}
  For every $H\colon I^3 \to M$, there exists $\epsilon>0$ such that $\dax(F)=\dax(H)$ whenever $F$ is $\epsilon$-close to $H$ and $F|_{\partial I^3}=H|_{\partial I^3}$.
  Indeed, since $M$ is an ANR, there is a homotopy $H\simeq F$ rel~$\partial I^3$ for sufficiently close $F$ \cite[Chapter~IV, Theorem~1.2]{Hu:1965-1}.
  The assertion follows from \cref{theorem:dax-homotopy-invariance}.
\end{remark}

For maps $H$,~$F\colon I^3\to M$ such that $H(s,t,1)=F(s,t,0)$ for $(s,t)\in I^2$, define the \emph{concatenation} $H*F\colon I^3\to M$ by $(H*F)\circ\rho_0=H$, $(H*F)\circ\rho_1=F$, where $\rho_i\colon I^3\to I^3$ is defined by $\rho_i(s,t,u)=(s,t,(u+i)/2)$ for $i=0,1$.
Also, define $H^{-1}\colon I^3\to M$ by $H^{-1}(s,t,u)=H(s,t,1-u)$.

Note that $H*F$ satisfies \cref{item:dax-boundary-condition} if $H$ and $F$ do, and $H^{-1}$ satisfies \cref{item:dax-boundary-condition} if $H$ does.
Similarly for~\cref{item:dax-boundary-interior-condition}.

We write $\dax_M$ when we need to specify the ambient 4-manifold~$M$.
Let $f\colon (M,\partial M) \hookrightarrow (N,\partial N)$ be an orientation preserving embedding into a 4-manifold~$N$.
Let $f_*\colon \Z[\pi_1(M)] \to \Z[\pi_1(N)]$ be the induced homomorphism, and let $f_\bullet \colon \Z[\pi_1(M) \sm 1] \to \Z[\pi_1(N) \sm 1]$ be the homomorphism defined by $f_\bullet(g)=0$ if $f_*(g)=1$, $f_\bullet(g)=f_*(g)$ otherwise, for $g\in \pi_1(M)$.

\begin{theorem}[Additivity, inversion and composition]
  \label{theorem:dax-additivity-inversion-composition}
  \leavevmode\Nopagebreak
  \begin{enumerate}
    \item \label{item:dax-additivity}
    $\dax(H*F)=\dax(H)+\dax(F)$.
    \item \label{item:dax-inversion}
    $\dax(H^{-1})=-\dax(H)$.
    \item \label{item:dax-composition}
    $\dax_N(f\circ H) = f_\bullet \dax_M(H)$ if $f\colon (M,\partial M) \to (N,\partial N)$ is an orientation preserving embedding.
    In particular, $\dax(f\circ H) = f_* \dax(H)$ if $f$ is an orientation preserving self-homeomorphism of $M$ such that $f(z_0)=z_0$, $f(z_1)=z_1$.
  \end{enumerate}
\end{theorem}

\begin{proof}
  To show \cref{item:dax-additivity}, we may assume that $H$ and $F$ satisfy~\cref{item:dax-boundary-interior-condition}, by \cref{lemma:interior-perturb,theorem:dax-homotopy-invariance}.
  Then $\Phi_{H*F}^{-1}(\Delta) = (\rho_0\times\rho_0)(\Phi_{H}^{-1}(\Delta)) \cup (\rho_1\times\rho_1)(\Phi_{F}^{-1}(\Delta))$ and $G_{H*F}=G_H\cup G_F$.
  Choose $V_g$ and $W_g$ as in the definition of $n(g)$ for $H$ and $F$, respectively.
  Set $V_g=\emptyset$ for $g\notin G_H$, and similarly for~$W_g$.
  The conclusion follows from computing $\dax(H*F)$ using $(\rho_0\times\rho_0)(V_g)\cup(\rho_1\times\rho_1)(W_g)$.
  A similar argument shows \cref{item:dax-inversion,item:dax-composition}.
\end{proof}

\begin{theorem}
  \label{theorem:dax-vanishes-for-injective-track}
  Suppose that $H\colon I^3 \to M$ is a map such that $s\mapsto H(s,\tau) \in M$ is an embedding for every $\tau\in I^2$, or equivalently, the track $\overline H$ is an embedding.
  Then $\dax(H)=0$.
\end{theorem}

\begin{proof}
  Since $\Phi_H^{-1}(\Delta) =\{(x,x)\mid x\in I^3\}$ and $g_H(x,x)=1$, $G_H=\emptyset$, and thus $\dax(H)=0$.
\end{proof}

\subsection{Dax homomorphism of \texorpdfstring{$\pi_3(M)$}{\pi₃(M)}}
\label{subsection:dax-homomorphism-pi_3}

As an application, we define a homomorphism of $\pi_3(M)$ which plays a key role in later sections.
Fix an embedding $\zeta\colon I\to M$ with $\zeta(s)=z_s$ for $s=0,1$.
Identify $\pi_3(M) = \pi_3(M,z_0)$ with the rel~$\partial I^3$ homotopy classes of maps $H\colon I^3\to M$ satisfying $H(s,\tau)=\zeta(s)$ for $(s,\tau)\in I\times \partial I^2$, $H(s,\tau)=z_s$
for $(s,\tau)\in \{0,1\}\times I^2$.
We say that $H$ is a representative of the element $[H]\in\pi_3(M)$.

\begin{definition}
  \label{definition:dax-homomorphism}
  Define the \emph{Dax homomorphism} $d_\zeta = d_{M,\zeta}\colon \pi_3(M) \to \Z[\pi_1(M)\sm 1]$ by $d_\zeta([H])=\dax(H)$ for a representative $H\colon I^3 \to M$.
\end{definition}

Note that a representative $H$ satisfies~\cref{item:dax-boundary-condition} since $\overline H|_{\partial I^3}=(\zeta\times\id_{I^2})|_{\partial I^3}$.
Thus $\dax(H)$ is defined.
By \cref{theorem:dax-homotopy-invariance}, $d_\zeta([H])$ is well-defined, independent of the choice of~$H$.
If $H$,~$F$ represent $[H]$, $[F]\in \pi_3(M)$, then $[H*F] = [H]+[F] \in \pi_3(M)$.
Thus we have
\[
  d_\zeta([H]+[F]) = \dax(H*F) = \dax(H)+\dax(F) = d_\zeta([H])+d_\zeta([F])
\]
by \cref{theorem:dax-additivity-inversion-composition}~\cref{item:dax-additivity}.
This shows that $d_\zeta$ is a homomorphism.

\section{Maps with a generic track}
\label{section:maps-with-generic-track}

In this section, we consider maps $H\colon I^3 \to M$ with a generic track, and show that our homological definition of $\dax(H)$ agrees with the $\pi_1$-labeled double point counting used in the smooth category.

\begin{definition}
  \label{definition:map-with-generic-track}
  We say that $H\colon I^3 \to M$ is a map with a \emph{generic track} if $H$ satisfies~\cref{item:dax-boundary-interior-condition} and $\overline H\colon I^3 \to M\times I^2$ is a (topological) embedding except at finitely many transverse double points.
\end{definition}

In \cref{subsection:computation-generic-track}, we describe how to compute of $\dax(H)$ for maps with a generic track, and using this, we show that our homological approach generalizes the smooth Dax self-intersection.
In \cref{subsection:application-dax-isomorphism}, we apply this computation to give a short proof of the topological Dax isomorphisms.

\subsection{Computation using a generic track}
\label{subsection:computation-generic-track}

Suppose $H$ is a map with a generic track.
For a double point $p=\overline H(x_p)=\overline H(y_p)$ with $s(x_p)<s(y_p)$, choose small 3-ball neighborhoods $A_p$,~$B_p\subset \inte(I^3)$ of $x_p$,~$y_p$ on which $\overline H$ is injective and orient them with the orientation of~$I^3$.
The product of the orientations of $\overline H(A_p)$ and $\overline H(B_p)$ gives a local orientation of $M\times I^2$ at~$p$.
Define $\epsilon(p)=1$ if it agrees with the preferred orientation of $M\times I^2$;
$\epsilon(p)=-1$ otherwise.
Let
\[
  g(p)=[H|_{[0,s(x_p)]\times\tau(x_p)}\cdot (H|_{[0,s(y_p)]\times\tau(y_p)})^{-1}] \in \pi_1(M).
\]

\begin{lemma}
  \label{lemma:dax-for-generic-track}
  If $H$ is a map with a generic track, then
  \[
    \dax(H) = \sum_{g(p)\ne 1} \epsilon(p)\cdot g(p).
  \]
\end{lemma}

\begin{proof}
  Shrinking $A_p$ and $B_p$ if necessary, assume that the products $A_p\times B_p\subset\inte(P)$ are pairwise disjoint and that $(A_p\times B_p)\cap\Phi_H^{-1}(\Delta)=\{(x_p,y_p)\}$.
  Then we can compute $\dax(H)$ using $V_g=\bigsqcup_{g(p)=g}A_p\times B_p$ for $g\ne1$.

  We claim that
  \[
    \bigl\langle u_\Delta, (\Phi_H|_{A_p\times B_p})_*[A_p\times B_p] \bigr\rangle = -\epsilon(p).
  \]
  To verify this, using the transversality hypothesis, choose local coordinates $\R^3\times \R^3$ near $p\in M\times I^2$ identifying the oriented sheets $\overline H(A_p)$ and $\overline H(B_p)$ with $\R^3\times0$ and $0\times\R^3$, respectively.
  This chart has orientation sign~$\epsilon(p)$.
  The corresponding product chart $(\R^3\times \R^3)^2$ near $(p,p)\in (M\times I^2)^2$ is orientation preserving and identifies $(\overline H \times \overline H)(A_p\times B_p)$ with $\R^3\times 0\times 0\times \R^3$, which intersects the oriented diagonal $\Delta_{\R^3\times \R^3}$ negatively.
  The claim follows from this.

  By excision, for every $g\ne1$,
  \[
    n(g) = - \bigl\langle u_\Delta, (\Phi_H|_{V_g})_*[V_g] \bigr\rangle
    = -\sum_{g(p)=g}\bigl\langle u_\Delta, (\Phi_H|_{A_p\times B_p})_*[A_p\times B_p] \bigr\rangle
    = \sum_{g(p)=g}\epsilon(p).
    \qedhere
  \]
\end{proof}

\subsubsection*{Comparison with the smooth Dax invariant}

Suppose that $M$ is a \emph{smooth} 4-manifold.
Fix a smooth neat base arc $\zeta\colon I\to M$.
Gabai defined the Dax self-intersection invariant $d(\alpha_{t,u}) \in \Z[\pi_1(M)\sm1]$ for a two-parameter family $\alpha_{t,u}\colon I\to M$ giving a nullhomotopy, through maps, of a loop of smooth embeddings based at~$\zeta$ \cite[proof of Theorem~3.7, Step~1]{Gabai:2021-1}.
See also \cite[Section~3]{Kosanovic-Teichner:2024-2}.
Also, in~\cite[Definition~4.1]{Cha-Kim:2023-1}, an invariant $\dax(H)\in \Z[\pi_1(M)\sm1]$ is defined essentially in the same way, for maps $H\colon I^3\to M$ which satisfy a smooth analog of~\cref{item:dax-boundary-condition}\@.
These invariants are defined by the double-point sum in \cref{lemma:dax-for-generic-track}, after Dax's generic perturbation \cite[Chapters III and~VII]{Dax:1972-1} that is based on Thom's multijet transversality theorems in the smooth category.

The following comparison is an immediate consequence of their definitions and \cref{lemma:dax-for-generic-track}.

\begin{corollary}
  \label{corollary:dax-smooth-comparison}
  Suppose that $M$ is smooth.
  For a smooth family $\{\alpha_{t,u}\}$, define the adjoint $H\colon I^3 \to M$ by $H(s,t,u)=\alpha_{t,u}(s)$.
  If Gabai's $d$ is defined for $\{\alpha_{t,u}\}$, then our $\dax(H)$ is defined and equal to Gabai's $d(\alpha_{t,u})$.
  Whenever the invariant $\dax(H)$ of~\cite{Cha-Kim:2023-1} is defined for a map~$H$, our $\dax(H)$ is defined and agrees with that of~\cite{Cha-Kim:2023-1}.
\end{corollary}


Gabai defined a homomorphism $d_3\colon \pi_3(M) \to \Z[\pi_1(M)\sm 1]$ using his $d(\alpha_{t,u})$ in~\cite[Definition~3.8]{Gabai:2021-1}.
The same homomorphism was denoted by $\dax_\zeta$ in~\cite[Definition~3.5]{Kosanovic-Teichner:2024-2}.
An analogous homomorphism $d_\zeta$ of $\pi_3(M)$ is defined in~\cite[Definition~4.2]{Cha-Kim:2023-1}, in the slightly more general setting in which $\zeta$ is not required to be neat.
\Cref{corollary:dax-smooth-comparison} implies the following.

\begin{corollary}
  \label{corollary:dax-homomorphism-smooth-comparison}
  The homomorphisms $d_3$,~$\dax_\zeta$ and $d_\zeta$ in \cite[Definition~3.8]{Gabai:2021-1}, \cite[Definition~3.5]{Kosanovic-Teichner:2024-2} and \cite[Definition~4.2]{Cha-Kim:2023-1} agree with our homomorphism $d_\zeta \colon \pi_3(M) \to \Z[\pi_1(M)\sm 1]$ in \cref{definition:dax-homomorphism}.
\end{corollary}

\subsection{Application to the Dax isomorphism}
\label{subsection:application-dax-isomorphism}

As before, fix distinct points $z_0$,~$z_1\in\partial M$.
Fix a locally flat neat embedding $\zeta\colon I\to M$ with $\zeta(i)=z_i$ for $i=0,1$.
When $M$ is smooth, assume that $\zeta$ is smooth.

In the definition of the spaces below, we assume that all maps and embeddings $f\colon I\to M$ satisfy $f^{-1}(\partial M)=\partial I$ and have the endpoint germ of~$\zeta$, meaning $f=\zeta$ on some neighborhood of~$\partial I$.
Let $\Map(I,M)$ be the space of all such continuous maps equipped with the compact-open topology.
Let $\Emb^0(I,M)\subset \Map(I,M)$ be the subspace of embeddings, without requiring local flatness.
Let $\Emb^t(I,M)$ be the space of locally flat embeddings, which is defined to be the geometric realization of the standard simplicial set model.
For details, refer to, e.g., Definitions~A.1 and~A.3 and the paragraph preceding Proposition~6.4 in~\cite{Gomez-Lopez-Kupers:2022-1}.
When $M$ is smooth, let $\Emb^s(I,M)$ be the space of smooth neat embeddings, equipped with the Whitney $C^\infty$ topology.
We use $\zeta$ as the basepoint of the above spaces.

There are natural homomorphisms on homotopy groups
\begin{equation*}
  \pi_1 \Emb^s(I,M) \to \pi_1 \Emb^t(I,M) \to \pi_1 \Emb^0(I,M).
\end{equation*}
We remark that the first map is defined using a simplicial model of the smooth embedding space, which is known to have the same weak homotopy type as $\Emb^s(I,M)$.
Refer to Convention~6.1 and the proof of Proposition~6.4 in \cite{Gomez-Lopez-Kupers:2022-1}.

Following the smooth category formulation in \cite[Definition~3.1]{Gabai:2021-1}, define
\[
  \pi_1^D\Emb^c(I,M) = \ker\{ \pi_1\Emb^c(I,M) \rightarrow \pi_1\Map(I,M) \}
\]
for $c=s$,~$t$,~$0$.
Note that the above natural homomorphisms restrict to
\[
  \pi_1^D \Emb^s(I,M) \to \pi_1^D \Emb^t(I,M) \to \pi_1^D \Emb^0(I,M).
\]

An element in $\pi_1^D\Emb^0(I,M)$ is represented by a based loop $\{f_t\colon I \hookrightarrow M\}_{0\le t\le 1}$ with $f_0=f_1=\zeta$.
Choose a based nullhomotopy of $\{f_t\}$ in $\Map(I,M)$, and let $H\colon I^3 \to M$ be its adjoint, so that $H(s,\tau)=\zeta(s)$ for $s\in I$, $\tau\in (I\times 0) \cup (\partial I\times I)$ and $H(s,t,1)=f_t(s)$ for $s$,~$t\in I$.
Note that $H$ satisfies~\cref{item:dax-boundary-condition}.
We call $H$ a \emph{filling} for~$\{f_t\}$.
Define
\[
  \phi_M^0\colon \pi_1^D\Emb^0(I,M) \to \Z[\pi_1(M)\sm 1] / d_\zeta(\pi_3(M))
\]
by $\phi_M^0([f_t])=\dax(H)$, where $d_\zeta\colon\pi_3(M) \to \Z[\pi_1(M)\sm 1]$ is the homomorphism defined in \cref{subsection:dax-homomorphism-pi_3}.
Define $\phi_M^t$ to be the composition of $\pi_1^D\Emb^t(I,M) \to \pi_1^D\Emb^0(I,M)$ and~$\phi_M^0$.
When $M$ is smooth, define $\phi_M^s$ to be the composition of $\pi_1^D\Emb^s(I,M) \to \pi_1^D\Emb^0(I,M)$ and~$\phi_M^0$.

\begin{lemma}
  The maps $\phi_M^0$, $\phi_M^t$, $\phi_M^s$ are well-defined homomorphisms.
\end{lemma}

\begin{proof}
  It suffices to prove the conclusion for~$\phi^0_M$.
  If $H$ and $F$ are fillings for $\{f_t\}$, then $H*F^{-1}$ represents an element $\alpha\in\pi_3(M,z_0)$, and thus $\dax(H)-\dax(F)=d_\zeta(\alpha)$ by \cref{theorem:dax-additivity-inversion-composition}.
  If $[f_t] = [g_t]$ in $\pi_1^D \Emb^0(I,M)$, then choose a based homotopy and let $G\colon I^3\to M$ be the adjoint such that $G(s,t,0)=f_t(s)$, $G(s,t,1)=g_t(s)$ and the arc $s\mapsto G(s,\tau)$ is injective for all $\tau\in I^2$.
  By \cref{theorem:dax-vanishes-for-injective-track}, $\dax(G)=0$.
  Thus, if $H$ is a filling for $\{f_t\}$, then $H*G$ is a filling for $\{g_t\}$, and $\phi_M^0([g_t]) = \dax(H*G) = \dax(H)+\dax(G) = \dax(H) = \phi_M^0([f_t])$. 

  If $H$ and $F$ are fillings of loops $\{f_t\}$ and $\{g_t\}$, then the concatenation $K$ of $H$ and $F$ along the $t$ coordinate is a filling of the product loop.
  The argument of the proof of \cref{theorem:dax-additivity-inversion-composition}~\cref{item:dax-additivity}, applied to the $t$-coordinate, gives $\dax(K)=\dax(H)+\dax(F)$.
  It follows that $\phi_M^0$ is a homomorphism.
\end{proof}

When $M$ is smooth, Gabai constructed an isomorphism
\[
  d\colon \pi_1^D \Emb^s(I,M) \to \Z[\pi_1(M)\sm 1] / d_3(\pi_3(M))
\]
in~\cite[Theorem~3.7]{Gabai:2021-1}.
Note that the codomain of $d$ is identical with that of our~$\phi_M^s$, since Gabai's $d_3$ on $\pi_3(M)$ agrees with our $d_\zeta$ by \cref{corollary:dax-homomorphism-smooth-comparison}.
Furthermore, by \cref{corollary:dax-smooth-comparison}, we have the following.

\begin{corollary}
  \label{corollary:identical-smooth-dax-isomorphism}
  The map $\phi_M^s$ is identical with the smooth Dax isomorphism $d$ in \cite[Theorem~3.7]{Gabai:2021-1}.
\end{corollary}

This gives a homological interpretation of the smooth Dax isomorphism.
Together with known smoothing techniques, it can be used to yield a short alternative proof of the topological Dax isomorphism theorem of Lin, Xie and Zhang~\cite[Theorem~1.1]{Lin-Xie-Zhang:2026-2}.

\begin{theorem}
  \label{theorem:dax-isomorphism}
  For $c=t$,~$0$, the homomorphism
  \[
    \phi_M^c\colon \pi_1^D\Emb^c(I,M) \to \Z[\pi_1(M)\sm1] / d_\zeta(\pi_3(M))
  \]
  is an isomorphism.
  Furthermore, the natural maps $\pi_1 \Emb^t(I,M) \to \pi_1 \Emb^0(I,M)$ and $\pi_1^D \Emb^t(I,M) \to \pi_1^D \Emb^0(I,M)$ are isomorphisms, and consequently $\phi_M^t$ and $\phi_M^0$ agree under these identifications.
  If $M$ is smooth, the same conclusion holds with $\Emb^t(I,M)$ and $\phi_M^t$ replaced by $\Emb^s(I,M)$ and~$\phi_M^s$.
\end{theorem}

\begin{proof}
  We may assume that $M$ is connected.
  Suppose $M$ is not necessarily smooth.
  We first reduce the general case to that of a smooth 4-manifold.
  Fix a point $q_0 \in \partial M \sm \{z_0,z_1\}$ and let $X=M\sm \{q_0\}$.
  We claim that $X$ has a smooth structure under which $\zeta$ is a smoothly embedded neat arc in~$X$.
  To prove this, choose a topological tubular neighborhood $\nu\zeta\cong I\times D^3$ in~$X$ with $\partial I\times D^3\subset \partial X$.
  Since the 4-manifold $X' = \overline{X\sm \nu\zeta}$ is connected and non-compact, $X'$ has a smooth structure by~\cite[Theorem~8.2]{Freedman-Quinn:1990-1}.
  We may assume that it restricts to the standard smooth structure on $I\times \partial D^3 \subset \partial X'$, since a smooth structure on a 3-manifold is unique up to isotopy.
  Thus the smooth structure on $X'$ and the standard smooth structure on $I\times D^3$ give rise to a smooth structure on~$X$, under which $\zeta$ is a smooth neat embedding.

  Note that, for $c=t$,~$0$, we have $\Emb^c(I,X)=\Emb^c(I,M)$ and $\Map(I,X) = \Map(I,M)$ and thus $\pi_1^D\Emb^c(I,X)=\pi_1^D\Emb^c(I,M)$.
  The inclusion induces the identifications $\pi_1(X)=\pi_1(M)$, $\Z[\pi_1(X)\sm 1] = \Z[\pi_1(M)\sm 1]$, $\pi_3(X)=\pi_3(M)$.
  Under these identifications, the definitions give $d_{X,\zeta} = d_{M,\zeta}$ and $\phi_X^c = \phi_M^c$.
  Therefore, it suffices to prove the conclusions for~$X$ instead of~$M$.

  Consider the following commutative diagram.
  \begin{equation*}
    \begin{tikzcd}[row sep=2ex,column sep=2em]
      \vphantom{\dfrac -I}
      & \pi_1^D \Emb^s(I,X) \ar[r,hook] \ar[d,"\beta^{s,t}"]
        \ar[ldd,"\phi_X^s"',start anchor=south west]
      & \pi_1 \Emb^s(I,X) \ar[r] \ar[d,"\alpha^{s,t}"]
      & \pi_1 \Map(I,X) \ar[d,equal]
      \\
      \vphantom{\dfrac II}
      & \pi_1^D \Emb^t(I,X) \ar[r,hook] \ar[d,"\beta^{t,0}"]
        \ar[ld,"\phi_X^t"' xshift=.25em,start anchor=south west]
      & \pi_1 \Emb^t(I,X) \ar[r] \ar[d,"\alpha^{t,0}"]
      & \pi_1 \Map(I,X) \ar[d,equal]
      \\
      \dfrac{\Z[\pi_1(M)\sm 1]}{d_\zeta(\pi_3(M))}
      & \pi_1^D \Emb^0(I,X) \ar[r,hook] \ar[l,"\phi_X^0"']
      & \pi_1 \Emb^0(I,X) \ar[r]
      & \pi_1 \Map(I,X)
    \end{tikzcd}
  \end{equation*}

  We claim that all the vertical arrows $\alpha^{s,t}$, $\alpha^{t,0}$, $\beta^{s,t}$, $\beta^{t,0}$ are isomorphisms.
  Indeed, it is known that $\alpha^{s,t}$ is surjective, e.g., see~\cite[Lemma~2.3]{Cha-Kim:2026-1}.
  The composition $\alpha^{t,0}\circ \alpha^{s,t}\colon \pi_1 \Emb^s(I,X) \to \pi_1 \Emb^0(I,X)$ is an isomorphism by Robinson's argument~\cite{Robinson:1974-1}.
  We note that Robinson states the result for closed manifolds, but her argument also applies to our nonempty boundary setup.
  Thus, $\alpha^{s,t}$ is injective and $\alpha^{t,0}$ is surjective, and consequently $\alpha^{s,t}$ and $\alpha^{t,0}$ are isomorphisms.
  By diagram chasing, it also follows that $\beta^{s,t}$ and $\beta^{t,0}$ are isomorphisms.

  By \cref{corollary:identical-smooth-dax-isomorphism} and \cite[Theorem~3.7]{Gabai:2021-1}, $\phi_X^s$ is an isomorphism.
  From this and from the above claim, it follows that $\phi_X^t$ and $\phi_X^0$ are isomorphisms.

  If $M$ is smooth, use $M$ as $X$ in the above argument, to obtain the desired conclusions.
\end{proof}

\section{Dax invariant of disks}
\label{section:dax-for-disks}

In this section, we apply our homological approach to define an invariant $\Dax(D_0,D_1)$ of embedded disks $D_0$,~$D_1$ homotopic rel boundary in a topological 4-manifold~$M$.

In~\cite{Cha-Kim:2023-1}, the topological Dax invariant $\Dax_{M,\Gamma}(D_0,D_1)$ was developed for locally flat embedded disks, using the stable smoothing theorem for disks~\cite[Theorem~D]{Cha-Kim:2023-1} and smooth category results in~\cite{Gabai:2021-1,Kosanovic-Teichner:2024-2}.
In particular, the construction of their $\Dax_{M,\Gamma}(D_0,D_1)$ depends on the choice of an almost smooth structure $\Gamma$ of~$M$ \cite[Definitions~4.6 and 4.13]{Cha-Kim:2023-1}.

Our homological approach provides a completely topological interpretation.
From this, it follows that the value of $\Dax_{M,\Gamma}(D_0,D_1)$ is independent of the choice~$\Gamma$.

\begin{corollary}
  \label{corollary:dax-disk-smooth-structure-independence}
  Let $\Gamma$,~$\Gamma'$ be almost smooth structures on~$M$ and let $D_0$,~$D_1$ be locally flat neat disks in $M$ with common boundary.
  Whenever the Dax invariants $\Dax_{M,\Gamma}(D_0,D_1)$ and $\Dax_{M,\Gamma'}(D_0,D_1)$ are defined, they agree.
\end{corollary}

Furthermore, our approach gives an extension to non-locally flat disks and invariance under non-locally flat isotopy.

We begin with a symmetry property of $\dax(H)$, which plays an essential role in the definition of the Dax invariant of disks.

\subsection{Symmetry of Dax invariants}
\label{subsection:symmetry}

Let $\sigma$ be the linear involution of $\Z[\pi_1(M)\sm 1]$ defined by $\sigma(g)=g^{-1}$ for $g\in \pi_1(M)\sm \{1\}$.
Let $\Z[\pi_1(M)\sm 1]^\sigma \subset \Z[\pi_1(M)\sm 1]$ be the subgroup of all elements fixed by~$\sigma$.
For $H\colon I^3 \to M$, define the \emph{reduced track} $\widehat H\colon I^3 \to M\times I$ by $\widehat H(s,t,u)=(H(s,t,u),u)$.

\begin{theorem}[Symmetry]
  \label{theorem:dax-symmetry}
  Suppose that $H\colon I^3\to M$ satisfies~\cref{item:dax-boundary-condition}, $\widehat H(\partial I^3) \subset \partial(M\times I)$, and $\widehat H(\partial I^3) \subset M\times I$ is $\pi_1$-null, meaning that the inclusion induces a trivial map on~$\pi_1$.
  Then $\dax(H) \in \Z[\pi_1(M)\sm 1]^\sigma$.
\end{theorem}

In the definition of $\dax(H)$, we used the subspace $P\subset I^3\times I^3$ to count only one of $(x,y)$ and $(y,x)$ for which $\overline H(x)=\overline H(y)$ holds.
If the whole domain $I^3\times I^3$ were used, it would give the $\pi_1$-labeled intersection of $\overline H$ and~$\overline H$ that counts both $(x,y)$ and $(y,x)$;
by investigating the vanishing of this intersection and its relationship with $\dax(H)$, we obtain the desired symmetry.
We remark that this approach was already applied to the Dax invariant in the smooth category in~\cite{Kosanovic-Teichner:2024-2}.
Our treatment below may be viewed as an adaptation to the topological category.

\begin{proof}[Proof of \cref{theorem:dax-symmetry}]
  By \cref{lemma:interior-perturb,theorem:dax-homotopy-invariance}, after a homotopy rel~$\partial I^3$, we may assume that $H(\inte(I^3))\subset\inte(M)$.
  Note that this implies $\widehat H(\inte(I^3)) \subset \inte(M\times I)$.

  For $\epsilon\in \{+,-\}$ and $0\le r\le 1$, define $\overline H{}^\epsilon_r\colon I^3 \to M\times I^2$ by
  \[
    \overline H{}_r^\epsilon(s,t,u) = (H(s,t,u),(1-r)\cdot t+r\cdot (\tfrac12+\epsilon\cdot\tfrac14),u).
  \]
  Note that $\overline H{}_r^+$ and $\overline H{}_r^-$ are deformations of $\overline H$ into the slices $M\times\tfrac34\times I$, $M\times\tfrac14\times I \subset M\times I^2$.

  To define homological intersection of $\overline H{}_r^+$ and $\overline H{}_r^-$ and relate it to the self-intersection of $\overline H = \overline H{}_0^+ = \overline H{}_0^-$, we proceed as follows.

  Let $Z=\{(x,y)\in I^3\times I^3:\widehat H(x)=\widehat H(y)\}$, a compact set.
  Define $\ell\colon Z\to \pi_1(M)$ by
  \[
    \ell(x,y)= \bigl[ H|_{[0,s(x)]\times\tau(x)} \cdot (H|_{[0,s(y)]\times\tau(y)})^{-1} \bigr] \in \pi_1(M).
  \]
  Since sufficiently close based loops are homotopic, $\ell$ is locally constant on~$Z$.

  We claim that $\ell=1$ on $Z\cap\partial(I^3\times I^3)$.
  Indeed, if $(x,y)\in Z\cap\partial(I^3\times I^3)$, then $x$,~$y\in \partial I^3$, since $\widehat H(\inte(I^3))\subset \inte(M\times I)$ and $\widehat H(\partial I^3) \subset \partial(M\times I)$.

  Choose paths $\alpha_x$,~$\alpha_y$ in $\partial I^3$ from $(0,\tau(x))$,~$(0,\tau(y))$ to $x$,~$y$, respectively.
  Since $I^3$ is simply connected, $\alpha_x$ and $[0,s(x)]\times\tau(x)$ are homotopic in $I^3$ rel endpoints, and similarly for $\alpha_y$ and $[0,s(y)]\times\tau(y)$.
  Thus the loop $(\widehat H\alpha_x)(\widehat H\alpha_y)^{-1}$ represents $\ell(x,y)$ in $\pi_1(M\times I)=\pi_1(M)$.
  It follows that $\ell(x,y)=1$, since $(\widehat H\alpha_x)(\widehat H\alpha_y)^{-1}$ lies in the $\pi_1$-null subset $\widehat H(\partial I^3) \subset M\times I$.

  Fix $g\ne1$.
  Then $\ell^{-1}(g) \subset \inte(I^3\times I^3)$ by the above claim.
  Recall that we have chosen $V_g\subset \inte(P)$ to define $\dax(H)$, using the fact that $g_H^{-1}(g)\subset \inte(P)$ and $g_H$ is locally constant.
  By the same argument applied to $\ell$ in place of $g_H$, choose a compact codimension-zero submanifold $W_g\subset\inte(I^3\times I^3)$ such that $W_g\cap Z=\ell^{-1}(g)\subset\inte(W_g)$.
  Let $[W_g]\in H_6(W_g,\partial W_g)$ be the fundamental class induced by the orientation of $I^3\times I^3$.

  Define $\Phi_r \colon I^3\times I^3 \to (M\times I^2)^2$ by $\Phi_r(x,y)=(\overline H{}_r^+(x), \overline H{}_r^-(y))$.
  Since $\Phi_r^{-1}(\Delta)\subset Z$ and $\partial W_g\cap Z=\emptyset$, $\Phi_r$ restricts to a homotopy of maps $(W_g,\partial W_g) \to ((M\times I^2)^2, (M\times I^2)^2 \sm \Delta)$.
  Let $u_\Delta\in H^6((M\times I^2)^2, (M\times I^2)^2\sm\Delta)$ be the Thom class for $\Delta$, which we use to define $\dax(H)$.
  For $r=1$, $\overline H{}_1^+$ and $\overline H{}_1^-$ have disjoint images, and thus $\Phi_1$ avoids~$\Delta$.
  It follows that $(\Phi_1|_{W_g})_*[W_g] = 0$.
  Furthermore, by homotopy invariance, we have
  \[
    \bigl\langle u_\Delta, (\Phi_0|_{W_g})_*[W_g] \bigr\rangle
    = \bigl\langle u_\Delta, (\Phi_1|_{W_g})_*[W_g] \bigr\rangle = 0.
  \]

  Suppose $g\ne g^{-1}$.
  Let $S$ and $T$ denote the maps $(a,b)\mapsto (b,a)$ on $I^3\times I^3$ and $(M\times I^2)^2$, respectively.
  Note that $\Phi_0|_P=\Phi_H$ and $\ell(y,x)=\ell(x,y)^{-1}$.
  Furthermore, every $(x,y)\in W_g\cap\Phi_0^{-1}(\Delta)$ satisfy $s(x)\ne s(y)$; 
  otherwise we would have $x=y$ and $\ell(x,y)=1$.
  From this, it follows that $W_g\cap\Phi_0^{-1}(\Delta) = g_H^{-1}(g)\sqcup S(g_H^{-1}(g^{-1}))$.
  Recall that, in the definition of $\dax(H)$, $V_g$ may be chosen in any given open neighborhood of $g_H^{-1}(g)$ in~$\inte(P)$.
  Using this, we may choose $V_g$ and $V_{g^{-1}}$ so that $V_g$ and $S(V_{g^{-1}})$ are disjoint and $V_g\sqcup S(V_{g^{-1}})\subset\inte(W_g)$.
  Then we have
  \[
    \begin{aligned}
      0 &= \bigl\langle u_\Delta, (\Phi_0|_{W_g})_*[W_g] \bigr\rangle
      \\
      &= \bigl\langle u_\Delta, (\Phi_0|_{V_g})_*[V_g] \bigr\rangle
        +\bigl\langle u_\Delta, (\Phi_0|_{S(V_{g^{-1}})})_*[S(V_{g^{-1}})] \bigr\rangle
      \\
      &= \bigl\langle u_\Delta, (\Phi_H|_{V_g})_*[V_g] \bigr\rangle
        -\bigl\langle u_\Delta, (\Phi_H|_{V_{g^{-1}}})_*[V_{g^{-1}}] \bigr\rangle
      = -n(g)+n(g^{-1}),
    \end{aligned}
  \]
  where the second equality follows by excision and the third from the facts that $\Phi_0\circ S=T\circ\Phi_0$, $S$ is orientation reversing, and $T^*u_\Delta = u_\Delta$.
  Therefore, $\dax(H)=\sigma\dax(H)$.
\end{proof}

\subsubsection*{Symmetry of the Dax homomorphism of~\texorpdfstring{$\pi_3(M)$}{{\pi₃(M)}}}

Let $\zeta\colon I\to M$ be an embedding with $\zeta(s)=z_s$ for $s=0,1$.
Recall that we have defined the homomorphism $d_\zeta \colon \pi_3(M) \to \Z[\pi_1(M)\sm 1]$ in \cref{definition:dax-homomorphism}.

\begin{corollary}
  \label{corollary:dax-homomorphism-symmetry}
  Suppose $\zeta\colon I \to \partial M$ is an embedding into $\partial M$.
  Then the homomorphism $d_\zeta$ of $\pi_3(M)$ takes values in $\Z[\pi_1(M)\sm 1]^\sigma$.
\end{corollary}

\begin{proof}
  Let $H\colon I^3\to M$ be a representative of an element in~$\pi_3(M)$.
  Then $\widehat H(\partial I^3) =\zeta(I)\times I \subset\partial M\times I \subset\partial(M\times I)$ and $\zeta(I)\times I\cong I^2$ is contractible.
  It follows that $\dax(H) \in \Z[\pi_1(M)\sm 1]^\sigma$ by \cref{theorem:dax-symmetry}.
\end{proof}

\subsection{Dax invariant of disks}
\label{subsection:dax-for-disk-definition-properties}

Let $D_0$,~$D_1\colon D^2\hookrightarrow M$ be neat embeddings which agree on $\partial D^2$ and are homotopic rel~$\partial D^2$.
Fix a smooth quotient map $\varphi\colon I^2\to D^2$ whose restriction to the interiors is an orientation-preserving diffeomorphism and such that the maps $s\mapsto\varphi(s,t)$, $t\in I$, are smooth embeddings of arcs from $(-1,0)$ to $(1,0)$ that foliate~$D^2$.
View $D_i$ as a map $I^2\to M$ by composing with~$\varphi$.

Let $z_0 = D_i(0,t)$, $z_1 = D_i(1,t)$.
Let $\zeta_-(s)=D_i(s,0)$ and $\zeta_+(s)=D_i(s,1)$, the common lower and upper semicircles of~$D_i$.
Note that $d_{\zeta_+}(\pi_3(M)) \subset \Z[\pi_1(M)\sm 1]^\sigma$ by \cref{corollary:dax-homomorphism-symmetry}.

Choose a homotopy $H\colon I^3\to M$ from $D_0$ to $D_1$ rel~$\partial I^2$.
Then $H$ satisfies~\cref{item:dax-boundary-condition}.
Moreover, $\widehat H(\partial I^3) = (D_0\times 0) \cup (\partial D_0 \times I) \cup (D_1\times 1) \subset \partial(M\times I)$ is an embedded 2-sphere, which is simply connected.
Thus, by \cref{theorem:dax-symmetry}, $\dax(H) \in \Z[\pi_1(M)\sm 1]^\sigma$.

\begin{definition}
  \label{definition:dax-disks}
  Define the \emph{Dax invariant} of $(D_0,D_1)$ by
  \[
    \Dax(D_0,D_1) = \dax(H) \in\Z[\pi_1(M)\sm 1]^\sigma / d_{\zeta_+}(\pi_3(M)).
  \]
\end{definition}

We note that the disks $D_0$,~$D_1$ are \emph{not required to be locally flat} in \cref{definition:dax-disks}.

With our homological definition of $\dax(H)$ and its properties established above, one can verify that $\Dax(D_0,D_1)$ is independent of the choice of~$H$ using arguments in \cite[Section~4]{Gabai:2021-1} and \cite[Section~4]{Kosanovic-Teichner:2024-2}.
For concreteness, we give details below.

\begin{lemma}
  The value $\Dax(D_0,D_1)$ is well-defined, independent of the choice of~$H$.
\end{lemma}

\begin{proof}
  Let $H$,~$F\colon D_0 \simeq D_1$ rel~$\partial I^2$ be homotopies.
  By \cref{lemma:interior-perturb,theorem:dax-homotopy-invariance}, we may assume that $H$ and $F$ satisfy~\cref{item:dax-boundary-interior-condition}.
  Put $L=H*F^{-1}$.
  Then $L(s,t,u)=D_0(s,t)$ on~$\partial I^3$.
  Also, $\dax(L)=\dax(H)-\dax(F)$ by \cref{theorem:dax-additivity-inversion-composition}~\cref{item:dax-additivity,item:dax-inversion}.

  The map $L$ can be converted to a map $K\colon I^3 \to M$ representing an element of~$\pi_3(M)$.
  Specifically, let $\cP(M)$ be the space of paths in $M$ from $z_0$ to $z_1$, equipped with the compact-open topology, and consider the adjoint $K^{\ad}\colon I^2\to\cP(M)$ defined by $K^{\ad}(t,u)(s)=K(s,t,u)$.
  Then the following diagram depicts $K^{\ad}$ for the representative $K$ associated with~$L$.
  \[
    K^{\ad} = \vcenter{\hbox{
    \begin{tikzpicture}[
      x=1cm, y=1cm,
      line width=.6pt,
      edge/.style={-stealth,shorten <=2.5*\r,shorten >=2.5*\r},
    ]
      \small
      \def\W{3.7}\def\b{1}
      \def\r{1pt}
      \fill (0,0) circle(\r) (\W,0) circle(\r) (0,\W) circle(\r) (\W,\W) circle(\r)
        (\b,\b) circle(\r) (\W-\b,\b) circle(\r) (\b,\W-\b) circle(\r) (\W-\b,\W-\b) circle(\r);
      \tiny
      \draw [edge,-] (0,0)--(\W,0) node[midway,below]{$\zeta_+$};
      \draw [edge,-] (0,\W)--(\W,\W) node[midway,above]{$\zeta_+$};
      \draw [edge,-] (0,0)--(0,\W) node[midway,left]{$\zeta_+$};
      \draw [edge,-] (\W,0)--(\W,\W) node[midway,right]{$\zeta_+$};
      \draw [edge,-] (\b,\b)--(\b,\W-\b) node[midway,left]{$\zeta_-$};
      \draw [edge,-] (\W-\b,\b)--(\W-\b,\W-\b) node[midway,right]{$\zeta_+$};
      \foreach \i in {\b,\W-\b} {
        \draw [edge] (\b,\i)--(\W-\b,\i) node[midway,fill=white]{$D_0$};
      }
      \draw [edge] (\b,\b)--(0,0) node[midway,fill=white]{$D_0$};
      \draw [edge] (\b,\W-\b)--(0,\W) node[midway,fill=white]{$D_0$};
      \draw [edge] (\b,\b)--(\W,0) node[midway,fill=white]{$D_0$};
      \draw [edge] (\b,\W-\b)--(\W,\W) node[midway,fill=white]{$D_0$};
      \draw [edge,-] (\W-\b,\b)--(\W,0) node[midway,fill=white]{$\zeta_+$};
      \draw [edge,-] (\W-\b,\W-\b)--(\W,\W) node[midway,fill=white]{$\zeta_+$};
      \small
      \draw (.5*\W,.5*\W) node{$L$};
    \end{tikzpicture}
    }}
  \]
  For concreteness, we provide an explicit description.
  Let $c\colon (\partial I^2)\times I \to I^2\sm\inte([\tfrac13,\tfrac23]^2)$ be the collar homeomorphism defined by $c(\mu,v)=(1-\frac{2}{3}v)\mu+\frac 13 v(1,1)$ for $\mu=(\mu_1,\mu_2)\in\partial I^2$, $v\in I$.
  Under~$c$, $(\partial I^2)\times0$ and $(\partial I^2)\times1$ correspond to $\partial I^2$ and $\partial[\tfrac13,\tfrac23]^2$, respectively.
  Define $K\colon I^3\to M$ by
  \[
    K(s,\tau)=
    \begin{cases}
      L(s,3\tau-(1,1)), & \tau\in[\tfrac13,\tfrac23]^2,\\
      D_0(s,(1-v)+v\mu_1), & \tau=c(\mu,v).
    \end{cases}
  \]  Then one can verify that $K$ represents an element $[K] \in \pi_3(M)$.

  For every $\tau\in I^2\sm\inte([\tfrac13,\tfrac23]^2)$, the arc $s\mapsto K(s,\tau)$ is embedded.
  It follows that $K$ satisfies \cref{item:dax-boundary-interior-condition} since so does~$L$.
  Furthermore, for $g\ne 1$, the set $g_K^{-1}(g)$ lies over $\inte([\tfrac13,\tfrac23]^2)$ and corresponds to $g_L^{-1}(g)$ under the affine reparametrization that preserves the orientations and $\pi_1$ labels.
  Therefore, $\dax(K)=\dax(L)$ by excision.
  Now we have
  \[
    \dax(H)-\dax(F)=\dax(L)=\dax(K)=d_{\zeta_+}([K]). \qedhere
  \]
\end{proof}

\begin{theorem}
  \label{theorem:dax-disk-addivity-inversion}
  For disks $D_0$,~$D_1$,~$D_2$ pairwise homotopic rel~$\partial$, the following holds.
  \begin{enumerate}
    \item\label{item:dax-disk-addivity}
    $\Dax(D_0,D_2)=\Dax(D_0,D_1)+\Dax(D_1,D_2)$.
    \item\label{item:dax-disk-inversion}
    $\Dax(D_1,D_0)=-\Dax(D_0,D_1)$.
  \end{enumerate}
\end{theorem}

\begin{proof}
  These are immediate consequences of \cref{theorem:dax-additivity-inversion-composition}~\cref{item:dax-additivity,item:dax-inversion}.
\end{proof}

\begin{theorem}
  \label{theorem:dax-disk-isotopy}
  The value of $\Dax(D_0,D_1)$ is invariant under isotopy of~$D_0$ and~$D_1$.
  If $D_0$ and $D_1$ are isotopic, then $\Dax(D_0,D_1)=0$.
\end{theorem}

We remark that neither the disks nor the isotopy in \cref{theorem:dax-disk-isotopy} are required to be locally flat.

\begin{proof}
  Suppose $H\colon D_0\approx D_1$ is an isotopy rel $\partial$.
  View $H$ as a homotopy $H\colon I^3 \to M$.
  Then, since $\overline H$ is an embedding, $\Dax(D_0,D_1)=\dax(H)=0$ by \cref{theorem:dax-vanishes-for-injective-track}.
  This shows the second assertion.
  The first assertion follows from this and \cref{theorem:dax-disk-addivity-inversion}.
\end{proof}

\subsubsection*{Comparison with the earlier developments for locally flat disks}

Let $\Gamma$ be an almost smooth structure of $M$ with singular point $p_0\in \inte M$, i.e., $\Gamma$ is a smooth structure on $M\sm p_0$.
In what follows, a topological neat disk $D\subset M$ is assumed to be smooth near $\partial D$ with respect to~$\Gamma$.

In~\cite[Definitions~4.6 and 4.13]{Cha-Kim:2023-1}, for \emph{locally flat} neat disks $D_0$,~$D_1\subset M$ such that $\partial D_0=\partial D_1$ and $D_0\simeq D_1$ rel~$\partial$, the invariant
\[
  \Dax_{M,\Gamma}(D_0,D_1)\in \Z[\pi_1(M)\sm 1]^\sigma / d_{M,\Gamma,\zeta_+}(\pi_3(M))
\]
is defined using stable disk smoothing and the smooth Dax invariant.
Here, $d_{M,\Gamma,\zeta_+}\colon\pi_3(M) \to \Z[\pi_1(M)\sm 1]^\sigma$ denotes the homomorphism given in~\cite[Lemma~4.12~(2)]{Cha-Kim:2023-1}, whose construction uses the almost smooth structure~$\Gamma$.
We begin by observing that this homomorphism can be reinterpreted by our topological method.

\begin{lemma}
  \label{lemma:identification-disk-dax-target}
  The homomorphism $d_{M,\Gamma,\zeta_+}$ in \cite[Lemma~4.12~(2)]{Cha-Kim:2023-1} agrees with our $d_{\zeta_+} = d_{M,\zeta_+}\colon \pi_3(M) \to \Z[\pi_1(M)\sm 1]^\sigma$ in \cref{definition:dax-homomorphism}.
  In particular, $d_{M,\Gamma,\zeta_+}$ is independent of the choice of~$\Gamma$.
\end{lemma}

\begin{proof}
  Identify $\pi_1(M\sm p_0)$ with $\pi_1(M)$ under the inclusion-induced isomorphism.
  In~\cite[Definition~4.2]{Cha-Kim:2023-1}, a homomorphism $d_{M\sm p_0,\Gamma,\zeta_+}\colon \pi_3(M\sm p_0) \to \Z[\pi_1(M)\sm 1]^\sigma$ is defined for the smooth 4-manifold $(M\sm p_0)_\Gamma$, and \cite[Lemma~4.12~(2)]{Cha-Kim:2023-1} shows that $d_{M\sm p_0,\Gamma,\zeta_+}$ induces a homomorphism $d_{M,\Gamma,\zeta_+}$ on the quotient $\pi_3(M)$ of~$\pi_3(M\sm p_0)$.
  By \cref{corollary:dax-homomorphism-smooth-comparison}, $d_{M\sm p_0,\Gamma,\zeta_+}$ agrees with the map $d_{M\sm p_0,\zeta_+}$ in \cref{definition:dax-homomorphism} on~$\pi_3(M\sm p_0)$.
  Since $d_{M,\zeta_+}$ is induced by $d_{M\sm p_0,\zeta_+}$, we have $d_{M,\Gamma,\zeta_+} = d_{M,\zeta_+}$ on~$\pi_3(M)$.
\end{proof}

From \cref{lemma:identification-disk-dax-target}, it follows that $\Dax_{M,\Gamma}(D_0,D_1)$ in \cite{Cha-Kim:2023-1} and our invariant $\Dax(D_0,D_1)$ in \cref{definition:dax-disks} live in the same abelian group $\Z[\pi_1(M)\sm 1]^\sigma / d_{\zeta_+}(\pi_3(M))$.

\begin{theorem}
  \label{theorem:dax-disk-topological}
  $\Dax_{M,\Gamma}(D_0,D_1) = \Dax(D_0,D_1)$ in $ \Z[\pi_1(M)\sm 1]^\sigma / d_{\zeta_+}(\pi_3(M))$.
\end{theorem}

\begin{proof}
  We briefly review the definition of $\Dax_{M,\Gamma}(D_0,D_1)$ in \cite{Cha-Kim:2023-1}.
  Fix a small smooth 4-ball $B$ in the smooth 4-manifold $(M\sm p_0)_\Gamma$.
  By topological isotopy, we may assume that $D_0$, $D_1 \subset M \sm (B\cup \{p_0\})$.
  By \cite[Theorem~D]{Cha-Kim:2023-1}, for some $k\ge 0$, the disk $D_i$ ($i=0,1$) is topologically isotopic to a smooth disk $D_i'$ in the smooth 4-manifold $M' = (M \sm \{p_0\})_\Gamma \csum k(S^2\times S^2)$, where the connected sum is formed at~$B$.
  Choose a homotopy $H\colon D_0 \simeq D_1$ rel $\partial$ in $M$ that avoids $p_0$ and $B$, and choose topological isotopies $h_i \colon D_i \simeq D_i'$ rel~$\partial$ in~$M'$.
  View $H$ and $h_i$ as maps $I^3\to M'$.
  Let $F = h_0^{-1} * H * h_1$ be the concatenation, which is a homotopy $D_0'\simeq D_1'$ rel~$\partial$ in~$M'$.
  Then, $\Dax_{M,\Gamma}(D_0,D_1)$ is defined to be the smooth Dax self-intersection of $F$ modulo $d_{\zeta_+}(\pi_3(M))$ \cite[Definitions~4.13, 4.6 and~4.1]{Cha-Kim:2023-1};
  its value is identical with our $\dax(F)$ in \cref{definition:dax-intersection-dbint,definition:dax-intersection}, by \cref{corollary:dax-smooth-comparison}.

  By \cref{theorem:dax-additivity-inversion-composition},
  \[
    \dax(F)=-\dax(h_0) + \dax(H) + \dax(h_1) \quad\text{in}\quad \Z[\pi_1(M')\sm 1]^\sigma =\Z[\pi_1(M)\sm 1]^\sigma,
  \]
  where $\dax(H)$ is computed in~$M'$.
  By \cref{theorem:dax-vanishes-for-injective-track}, $\dax(h_i)=0$ for $i=0,1$.
  Since $H$ can be viewed as a homotopy in $M$, $\dax(H)$ can also be computed in $M$ to yield the same value.
  Thus,
  \[
    \Dax_{M,\Gamma}(D_0,D_1)=\dax(H)=\Dax(D_0,D_1) \quad\text{in}\quad \Z[\pi_1(M)\sm 1]^\sigma / d_{\zeta_+}(\pi_3(M)).
    \qedhere
  \]
\end{proof}

\Cref{corollary:dax-disk-smooth-structure-independence} is an immediate consequence of \cref{theorem:dax-disk-topological}.

\subsubsection*{Non-locally flat isotopy and locally flat isotopy}

Recall that a neat disk $D\subset M$ has a \emph{boundary geometric dual} if there is a map $G\colon S^2\to \partial M$ which intersects $\partial D$ at exactly one point, transversely.
Under the presence of a boundary geometric dual, the invariance of the Dax invariant under non-locally flat isotopy has the following consequence.

\begin{corollary}
  \label{corollary:disk-isotopy-implies-locally-flat-isotopy}
  Let $M$ be an orientable topological 4-manifold with nonempty boundary.
  For locally flat neat disks in $M$ admitting a boundary geometric dual, isotopy implies locally flat isotopy.
  If $M$ is smooth, then for smooth neat disks in $M$ admitting a geometric dual in the boundary, topological isotopy implies smooth isotopy.
\end{corollary}

\begin{proof}
  Suppose that $D_0$ and $D_1$ are locally flat disks which are isotopic in $M$ and admit a boundary geometric dual.
  By \cref{theorem:dax-disk-isotopy}, $\Dax(D_0,D_1)=0$.
  Since $D_0$ and $D_1$ are homotopic rel~$\partial$, and since our $\Dax(D_0,D_1)$ agrees with that in~\cite{Cha-Kim:2023-1} by \cref{theorem:dax-disk-topological}, this implies that $D_0$ and $D_1$ are locally flat isotopic by~\cite[Corollary~H]{Cha-Kim:2023-1}.
  The smooth isotopy conclusion readily follows from this by~\cite[Corollary~A]{Cha-Kim:2023-1}.
\end{proof}

We record another application of the Dax invariant of disks.
In what follows, local knotting means adding a locally flat spherical to a locally flat surface, or a smooth spherical knot to a smooth surface.

\begin{corollary}
  \label{corollary:local-knotting-implies-locally-flat-isotopy}
  Let $\Sigma$ be a connected, locally flat neat surface in an orientable topological 4-manifold~$M$ admitting a boundary geometric dual.
  Then every surface obtained by local knotting on $\Sigma$ is ambient isotopic to~$\Sigma$.
  If $M$ is smooth and $\Sigma$ is smoothly embedded, then surfaces obtained by smooth local knotting on $\Sigma$ are smoothly isotopic.
\end{corollary}

\begin{proof}
  Suppose $\Sigma'\subset M$ is obtained from $\Sigma$ by local knotting.
  Fix a disk $D\subset \Sigma$ that intersects $\partial \Sigma$ in a semicircle of~$\partial D$, and let $X=\overline{M\sm \nu\Sigma|_{\Sigma\sm D}}$, the exterior of a topological normal bundle of~$\Sigma\sm D$.
  Note that $D\subset X$ is a neat disk,
  We may assume that
  $\Sigma'=(\Sigma\sm D)\cup D'$ where $D'=\Sigma'\cap X$ is obtained from $D$ by local knotting.
  We may also assume that a boundary geometric dual of $\Sigma$ meets~$D$ so that it is also a boundary geometric dual of~$D$.
  Since $D$ and $D'$ differ by local knotting, they are isotopic, and thus locally flat isotopic in $X$ by~\cref{corollary:disk-isotopy-implies-locally-flat-isotopy}.
  Extend it by the identity on $\Sigma\sm D$, to obtain a locally flat isotopy between $\Sigma$ and $\Sigma'$ in~$M$.
  By~\cite{Edwards-Kirby:1971-1}, $\Sigma$ and $\Sigma'$ are ambient isotopic.
  In the smooth case, the same argument works by using the smooth part of ~\cref{corollary:disk-isotopy-implies-locally-flat-isotopy}.
\end{proof}

\bibliographystyle{amsalpha}
\def\MR#1{}
\bibliography{research}

\end{document}